\documentclass[12pt]{article}
\usepackage{curves}
\usepackage{amsmath, amssymb, lamsarrow}

\usepackage{indentfirst,latexsym,bm}
\usepackage{amscd}
\usepackage{amsthm}
\usepackage{amsfonts}
\usepackage[dvips,draft]{graphics}
\usepackage{enumerate}

\def\mineappendix{
        \setcounter{section}{1}
        \setcounter{subsection}{0}
        \def\thesection{\Alph{section}}
        \def\sectionap{\@startsection  {section}{1}{\z@}
                        {-3.5ex plus-1ex minus-.2ex} {0ex plus.2ex}
                        {\reset@font\Large\bf  Appendix:  \, }
                        }
        }
\makeatother
\def\Proclaim #1. #2\par{\bigbreak\noindent{\sc#1.\enspace}{\it#2}\par}

\allowdisplaybreaks

\newtheorem{theorem}{Theorem}[section]

\newtheorem{lemma}[theorem]{Lemma}
\newtheorem{proposition}[theorem]{Proposition}
\newtheorem{corollary}[theorem]{Corollary}

\title{Connecting Hodge integrals to Gromov-Witten invariants by Virasoro operators}

\author{Xiaobo Liu \thanks{Research of the first author was partially supported by  NSFC research fund 11431001 and NSFC Tianyuan special funds 11626241.}, \hspace{6pt} Haijiang Yu}

\date{}

\begin{document}

\maketitle

\begin{abstract}
In this paper, we show that the generating function for linear Hodge integrals over moduli spaces of stable maps to a nonsingular projective variety $X$ can be connected to the generating function for Gromov-Witten invariants of $X$ by a series of differential operators $\{ L_m \mid m \geq 1 \}$ after a suitable change of variables. These operators satisfy the Virasoro bracket relation and can be seen as a generalization of the Virasoro operators appeared in the Virasoro constraints for Kontsevich-Witten tau-function in the point case. This result is an extension of the work in \cite{LW} for the point case which solved a conjecture of Alexandrov.
\end{abstract}
%

\section{Introduction}

It is well known that the generating function for intersection numbers of psi-classes on moduli spaces of stable curves
$\overline{\cal M}_{g,n}$ is
a tau function of the KdV hierarchy. This was conjectured by Witten and proved by Kontsevich (cf. \cite{Ko}). It was also proved
by Kazarian that the generating function of linear Hodge integrals over $\overline{\cal M}_{g,n}$ is a tau
function of KP hierarchy(cf. \cite{Ka}).
Motivated by the fact that the space of
all tau functions of KP hierarchy is a homogeneous space over which the group $\widehat{GL}(\infty)$ acts transitively,
 Alexandrov conjectured that the Hodge tau-function and Kontsevich-Witten tau-function can be connected by a linear combination
 of Virasoro operators \cite{Ale2014}. He also verified a weak version of this conjecture in \cite{Ale2015}.
The first author of this paper and Gehao Wang proved Alexandrov's conjecture in \cite{LW}. This result gives an explicit
$\widehat{GL}(\infty)$ operator which connecting these two tau functions arising from the intersection theory
 over the moduli space of curves  $\overline{\cal M}_{g,n}$. The main purpose of this paper is to extend
this result to moduli spaces $\overline{\cal M}_{g,n}(X, A)$ of degree $A \in H_2(X, \mathbb{Z})$ stable maps  with any nonsingular projective variety $X$ as the target.

Hodge integrals over $\overline{\cal M}_{g,n}(X, A)$ were introduced by Faber and Pandharipande in \cite{FP2000}. They are generalizations of Gromov-Witten invariants. It was proven in \cite{FP2000} that Hodge integrals can be reconstructed from Gromov-Witten invariants.
In \cite{DLYZ}, the authors gave an algorithm to represent the generating functions of Hodge integrals in terms of Gromov-Witten potentials and genus-0 primary two point functions. They also introduced Hodge potentials associated to semisimple Frobenius manifolds whose exponential is a tau function of a new integrable hierarchy, called Hodge hierarchy.  We hope that results in this paper may be useful in understanding structures of tau functions for the Hodge hierarchy.

Another interest for studying Hodge integrals is that their universal equations might be simpler. It is well known that relations
among  $\psi$-classes and boundary classes in the tautological ring of $\overline{\cal M}_{g,n}$ produce universal equations for Gromov-Witten invariants. Such equations play important role in studying properties for Gromov-Witten invariants. For example they are crucial ingredients in the study of low genus Virasoro conjecture. Usually equations from low degree tautological relations are more powerful than those from higher degree tautological relations.  From results in \cite{Io}, \cite{FP2005}, and \cite{GV}, we know there should exist degree $g$ tautological relations on $\overline{\cal M}_{g,n}$. But currently we do not have explicit formula for degree $g$
relations among only $\psi$-classes and boundary classes when $g>4$. Moreover known relations in low genera are also very complicated. However simple formulas might be obtained if Hodge classes are allowed. For example, a recursive degree $g$ formula including Hodge classes and $\psi$-classes for all genera was found in \cite{AS}. This should give universal equations for generating functions of Hodge integrals. It is interesting to investigate what are consequences of such kind of equations for Gromov-Witten invariants. This requires better understanding of relations between
Hodge integrals and Gromov-Witten invariants. We hope that results in this paper will also shed new lights in this direction.

Let $F_{g,H}^{X}(\mathbf{t},u)$ be the generating function of genus $g$ linear Hodge integrals over moduli spaces of
stable maps with target $X$. Precise definition of $F_{g,H}^{X}(\mathbf{t},u)$ will be given in equation \eqref{F_gH^X}.
Set $Z_H^X(\mathbf{t},u) = \exp(\sum\limits_{g\geq 0}\hbar^{g-1}F_{g,H}^{X})$ where $\hbar$ is a parameter.
Then $Z^X(\mathbf{t}) := Z_H^X(\mathbf{t},0)$ is the generating function of descendent Gromov-Witten invariants of $X$.
After a suitable change of variables $ \mathbf{t}\rightarrow (u,\mathbf{q}) $ which will be given in Section~\ref{sec:change_variable},
 $Z_H^X$ and $Z^X$ will become functions of $u$ and $\mathbf{q}=(q_n^{\alpha} \mid n \in \mathbb{Z}_{>0}, \alpha=0, 1, \ldots, N)$
 where $N+1$ is the dimension of the space of cohomology classes of $X$.
These functions will be denoted by $Z_H^X (\mathbf{q},u)$ and $Z^X(\mathbf{q})$ respectively. The main result of this paper is that these
two functions can be connected by a sequence of differential operators of the form
\begin{equation}\label{Lm}
L_m := \sum\limits_{k>0} \sum_{\alpha} (k+m)q_k^{\alpha}\frac{\partial}{\partial q_{k+m}^{\alpha}}
   +\frac{\hbar}{2}\sum\limits_{a+b=m} \sum_{\mu, \nu} ab \eta^{\mu\nu} \frac{\partial^2}{\partial q_a^{\mu}\partial q_b^{\nu}}
\end{equation}
for $m\geq 1$, where $\eta^{\mu\nu}$ are entries for the inverse matrix of the intersection pairing on $X$.
We can check that these operators satisfy the Virasoro bracket relation
$$[ L_m, L_n ] = (m-n)L_{m+n} $$
for all $m, n \geq 1$ (cf. Appendix \ref{app1}).
Actually $L_m$ can be seen as a generalization of the Virasoro operators used in \cite{LW} in the point case.
After a slight shift, these operators also annihilate the Kontsevich-Witten tau function $Z^X$ when $X$ is a point.
The precise statement for the main result of this paper is the following

\begin{theorem}\label{theorem}
The generating functions of linear Hodge integrals and descendent Gromov-Witten invariants can be connected as follows
$$Z_H^X (\mathbf{q},u) = \exp{(\sum\limits_{m\geq 1} a_m u^m L_m)} \cdot Z^X(\mathbf{q}) $$
where  $a_m$ are constants determined by
\begin{equation}\label{a_m}
\exp(\sum_{m>0}a_mz^{1-m}\frac{\partial}{\partial z})\cdot z =\left(-2\log(1-\frac{1}{1+z})-\frac{2}{1+z}  \right)^{-\frac{1}{2}}.
\end{equation}
\end{theorem}

The above theorem can be seen as a generalization of the main result of \cite{LW} which proves Alexandrov's conjecture.
One important ingredient in the proof of \cite{LW} is the Virasoro constraints for Kontsevich-Witten tau-function. This can not be
generalized to arbitrary projective manifold $X$ for two reasons. Firstly, the existence of Virasoro constraints for Gromov-Witten invariants of $X$ is currently only a conjecture, knowing as the Virasoro conjecture (cf. \cite{EHX}). This conjecture has not been proved for general cases. Secondly, the Virasoro operators $L_m$ in this paper are different from the Virasoro operators appeared in the Virasoro conjecture since the latter operators depend on the first Chern class of $X$. Therefore in the proof of Theorem \ref{theorem}, we have to
find a way to bypass the Virasoro constraints. To achieve this, we use a change of variable which is different from that in \cite{LW} even for the point case.

This paper is organized as follows. In Section \ref{sec:pre}, we review the definition of Hodge integrals and  Zassenhaus formula. In Section~\ref{sec:decompW}, we apply Zassenhaus formula to decompose the operator $\exp W_{\omega}$ which connecting $Z^X$ and generating functions of Hodge integrals.  It is interesting to see that Givental's quantization formula arise naturally in this process. Moreover we will also prove a formula (see Corollary \ref{cor}) which is a generalization of Givental's lemma in Section 2.3 of \cite{Gi1}. In Section~\ref{sec:change_variable}, we give the change of variables. In Section~\ref{sec:decompVir}, we decompose exponential of linear
combinations of Virasoro operators using Zassenhaus formula. Finally we prove Theorem \ref{theorem} in Section~\ref{sec:proofThm}.

\section{Preliminaries}
\label{sec:pre}

\subsection{Hodge integral}

Let $X$ be a nonsingular projective variety, $\{ \gamma_0, \gamma_1,\cdots ,\gamma_N \}$ be a graded $\mathbb{Q}$-basis of $H^{*}(X,\mathbb{Q})$, and take $\gamma_0$ to be the identity element. Let $\eta_{\mu\nu} = \int_{X} \gamma_{\mu} \cup \gamma_{\nu}$, and $\eta^{\mu\nu}$ be entries of the matrix which is the inverse of matrix $(\eta_{\alpha \beta})_{(N+1)\times(N+1)}$.

Let $\overline{M}_{g,n}(X,A)$ be the moduli space of stable maps
$f:C\longrightarrow X$ where $(C; p_1,\cdots,p_n)$ is a genus-$g$ nodal curve with $n$ marked points and
$f_*([C])=A \in H_2(X;\mathbb{Z})$. The Hodge bundle $\mathbb{E}$ over $\overline{M}_{g,n}(X,A)$ is the rank $g$ vector bundle with fiber $H^0(C,\omega_C)$ over $(C; p_1,\cdots,p_n;f) \in \overline{M}_{g,n}(X,A)$, where $\omega_C$ is the dualizing sheaf of $C$. Let $ch(\mathbb{E})$ be the Chern character of $\mathbb{E}$ and its degree $k$ component is denoted by $ch_k(\mathbb{E})$. By Mumford's relations (cf. \cite{Mu}), the even components of $ch(\mathbb{E})$ vanish (for all genera). So $ch(\mathbb{E}) = \sum\limits_{l \geq 1} ch_{2l-1}(\mathbb{E})$.

In \cite{FP2000}, Faber and Pandharipande considered {\bf Hodge integrals} of the following form
$$\langle \prod\limits_{j=1}^m ch_{l_j}(\mathbb{E}) \prod\limits_{i=1}^n \tau_{k_i}(\gamma_{\alpha_i}) \rangle_{g,A}^X = \int_{[\overline{M}_{g,n}(X,A)]^{vir}} \prod_{i=1}^n \psi_i^{k_i} \cup ev_i^*(\gamma_{\alpha_i}) \cup \prod\limits_{j=1}^m ch_{l_j}(\mathbb{E}) ,$$
where $[\overline{M}_{g,n}(X,A)]^{vir}$ is the virtual fundamental class of $\overline{M}_{g,n}(X,A)$,
$\psi_i \in H^2(\overline{M}_{g,n}(X,A);\mathbb{Q})$ is the cotangent line class and  $ev_i: \overline{M}_{g,n}(X,A) \longrightarrow X$ is the evaluation map defined by the $i$-th marked point on domain curves. When all classes $ch_{l_j}(\mathbb{E})$ are removed from the
above formula, we get the definition for {\bf descendant Gromov-Witten invariants}
$\langle  \prod\limits_{i=1}^n \tau_{k_i}(\gamma_{\alpha_i}) \rangle_{g,A}^X$.

Let $\mathbf{s}, \mathbf{t}$ be the sets of variables $\{s_n\}$, $\{t_n^{\alpha}\}$  respectively, where the range for $n$ is the set of
all non-negative integers $\mathbb{Z}_{\geq 0}$ and the range of $\alpha$ is $\{0, 1, \ldots, N\}$.
The generating function of genus-$g$ Hodge integrals is defined to be
\[ F_{g,\mathbb{E}}^{X}(\mathbf{t},\mathbf{s}) := \sum\limits_{A \in H_2(X,\mathbb{Z})} Q^{A}
  \langle \exp(\sum_{l \geq 1} s_{2l-1} ch_{2l-1}(\mathbb{E}) ) \exp (\sum\limits_{i=0}^{\infty}\sum\limits_{\alpha=0}^N t^{\alpha}_i \tau_i(\gamma_{\alpha})) \rangle_{g,A}^{X}
\]
where $Q^A$ is an element in the Novikov ring.
Let $ F_{\mathbb{E}}^{X} = \sum\limits_{g \geq 0} \hbar^{g-1} F_{g,\mathbb{E}}^{X}$, where $\hbar$ is a parameter, and set
 $Z_{\mathbb{E}}^X = \exp(F_{\mathbb{E}}^{X})$. When $\mathbf{s}=0$, we get the generating function for descendent Gromov-Witten
 invariants $Z^X(\mathbf{t} )= Z_{\mathbb{E}}^X ( \mathbf{t}, 0)$.

 Using Mumford's Grothendieck-Riemann-Roch formula (cf. \cite{Mu}), it was proved in \cite{FP2000} that $Z_{\mathbb{E}}^X$ satisfies the
   following systems of differential equations: For all $l \geq 1$,
\begin{equation}
\label{DZ=0}
\frac{\partial}{\partial s_{2l-1}} Z_{\mathbb{E}}^X =  \frac{B_{2l}}{(2l)!} D_l \cdot Z_{\mathbb{E}}^X,
\end{equation}
where
\begin{equation} \label{Dl}
 D_l := \frac{\partial}{\partial t^0_{2l}} - \sum\limits_{i=0}^{\infty}\sum\limits_{\alpha=0}^N t^{\alpha}_i\frac{\partial}{\partial t^{\alpha}_{i+2l-1}} + \frac{\hbar}{2} \sum\limits_{i=0}^{2l-2} \sum_{\mu, \nu} (-1)^{i}\eta^{\mu\nu} \frac{\partial}{\partial t^{\mu}_{i}} \frac{\partial}{\partial t^{\nu}_{2l-2-i}}
 \end{equation}
 and $B_{2l}$ are the Bernoulli numbers defined by
$$\frac{x}{e^x-1} = \sum\limits_{m=0}^{\infty} B_{m}\frac{x^m}{m!}.$$

From Equation~\eqref{DZ=0}, one can deduce that
\begin{equation}
\label{Z=e^WZ}
Z_{\mathbb{E}}^X(\mathbf{t},\mathbf{s}) = \exp (\sum\limits_{l\geq 1} \frac{B_{2l}}{(2l)!} \, s_{2l-1} \, D_{l} ) \cdot Z^X(\mathbf{t}).
\end{equation}
This formula is true since both sides  satisfy the system of differential equations \eqref{DZ=0} and have the same initial condition
at $\mathbf{s}=0$ (cf. \cite{Gi1},\cite{Gi2} and \cite{DLYZ}).

{\bf Linear Hodge integrals} are integrals of the following form
$$\langle \lambda_j  \prod\limits_{i=1}^n \tau_{k_i}(\gamma_{\alpha_i}) \rangle_{g,A}^X
:= \int_{[\overline{M}_{g,n}(X,A)]^{vir}} \prod_{i=1}^n \psi_i^{k_i} \cup ev_i^*(\gamma_{\alpha_i}) \cup \lambda_j ,$$
where $\lambda_j = c_j (\mathbb{E})$ is the $j$-th Chern class of $\mathbb{E}$. The generating function of genus-$g$ linear
Hodge integrals is given by
\begin{equation} \label{F_gH^X}
F_{g,H}^{X}(\mathbf{t},u) = \sum\limits_{A \in H_2(X,\mathbb{Z})} Q^{A} \langle \lambda(-u^2) \exp (\sum\limits_{i=0}^{\infty}\sum\limits_{\alpha=0}^N t^{\alpha}_i \tau_i(\gamma_{\alpha})) \rangle_{g,A}^{X}
\end{equation}
where $\lambda(y) := \sum\limits_{j\geq 0} \lambda_{j}y^j$ is the Chern polynomial of $\mathbb{E}$, and $u$ is a parameter.
Let $F_{H}^{X} := \sum\limits_{g \geq 0} \hbar^{g-1} F_{g,H}^{X}$ and
$Z_H^X := \exp(F_{H}^{X})$. When $u=0$, we also recover the generating function
for descendent Gromov-Witten invariants $Z^X(\mathbf{t})= Z^X_H(\mathbf{t},0)$.

When $X$ is a point, the corresponding functions $Z^{pt}(\mathbf{t})$ is a tau function of the KdV hierarchy (cf. \cite{Ko}) and $Z_H^{pt}(\mathbf{t},u)$ is a tau function of the KP hierarchy (cf. \cite{Ka}). Alexandrov's conjecture proved in \cite{LW} asserts that
these two tau functions can be connected by Virasoro operators.

By definition of Chern characters, we have (cf. example 4.4 in \cite{DLYZ})
\begin{equation} \label{lambdach}
\lambda(-u^2)  = \exp( \sum\limits_{l\geq 1} -(2l-2)! u^{2(2l-1)} ch_{2l-1} ).
\end{equation}
Hence $Z_{g,H}^{X}(\mathbf{t},u)$ is equal to $Z_{g,\mathbb{E}}^{X}(\mathbf{t},\mathbf{s})$ after the substitution
\[ s_{2l-1} = -(2l-2)! u^{2(2l-1)} \]
for all $l \geq 1$.

By Equation~\eqref{Z=e^WZ}, we have

\begin{equation}\label{eqn:lem2}
Z^{X}_{H}(\mathbf{t},u) = e^{W_u}\cdot Z^X(\mathbf{t})
\end{equation}
where $W_u$ is defined by
\begin{equation} \label{eqn:Wu}
W_u := -\sum\limits_{l\geq1} \frac{B_{2l}}{2l(2l-1)} u^{2(2l-1)} \, D_{l}.
\end{equation}

\subsection{Zassenhaus formula}

As in \cite{LW}, a crucial ingredient in the proof of Theorem~\ref{theorem} is the
Zassenhaus formula
 which can be considered as the inverse of Baker-Campbell-Hausdorff formula (cf. \cite{CMN}). For two elements $X$ and $Y$ in a Lie algebra $\mathfrak{g}$ over a field of characteristic zero, Zassenhaus formula expands $e^{\alpha(X+Y)}$ as
$$e^{\alpha(X+Y)} = e^{\alpha X} e^{\alpha K_1} e^{\alpha^2 K_2} e^{\alpha^3 K_3}\cdots,$$
where $K_n$ is a homogeneous Lie polynomial in $X$ and $Y$ of degree $n$. $K_1 = Y$ and
$$ K_n = \frac{1}{n}\sum\limits_{i=0}^{n-2} \frac{(-1)^{n-1}}{i!(n-1-i)!} ad_Y^i ad_X^{n-1-i}Y$$
for $n \geq 2$.

In particular, if  $\mathfrak{g}_1$ and $\mathfrak{g}_2$ are Lie subalgebras of $\mathfrak{g}$ satisfying the condition:
$$[\mathfrak{g}_1,\mathfrak{g}_2] \subseteq \mathfrak{g}_2 \quad and \quad [\mathfrak{g}_2,\mathfrak{g}_2] = 0,$$
then for any $X\in \mathfrak{g}_1, Y\in \mathfrak{g}_2,$
\begin{equation}\label{Zassenhaus}
\exp(X+Y) = \exp(X)\exp(\sum\limits_{n=1}^{\infty} \frac{(-1)^{n-1}}{n!} ad_X^{n-1} Y)
\end{equation}
(cf. \cite{LW}).
For example, the operators we mainly concern in this  paper come from the following three types:
$$\mathfrak{g}_1 = \left\{ \left. \sum\limits_{i,j,\mu} a_{ij}^{\mu} t_i^{\mu} \frac{\partial}{\partial t_{i+j}^{\mu}} \right| a_{ij}^{\mu} \ {\rm are \ constants} \right\},$$
$$\mathfrak{g}_2 = \left\{  \left. \sum\limits_{i,\mu} b_{i}^{\mu} \frac{\partial}{\partial t_{i}^{\mu}} \right| b_{i}^{\mu} \ {\rm are \ constants} \right\},$$
$$\mathfrak{g}_2' = \left\{  \left. \sum\limits_{i,j,\mu,\nu} c_{ij}^{\mu\nu} \frac{\partial^2}{\partial t_{i}^{\mu}t_{j}^{\nu}} \right| c_{ij}^{\mu\nu} \ {\rm are \ constants} \right\}.$$
We can check both $\mathfrak{g}_1, \mathfrak{g}_2 \ and \ \mathfrak{g}_1, \mathfrak{g}_2'$ satisfy the above condition.

\section{Decomposition of $W_{\omega}$}
\label{sec:decompW}

Motivated by both equations \eqref{Z=e^WZ} and \eqref{eqn:lem2}, we consider the operator
\begin{equation} \label{denW}
W_{\omega} := \sum_{l=1}^{\infty} \omega_l \, D_l
\end{equation}
where $\omega=(\omega_1, \omega_2, \ldots)$ are parameters. As in \cite{LW}, a crucial step in the proof of Theorem~\ref{theorem}
is to decompose operator $\exp(W_{\omega})$ using
Zassenhaus formula.  It is interesting to see that Givental's quantization operator (cf. \cite{Gi2})  appears naturally in this process.
The main result of this section is Proposition~\ref{e^W=e^Be^Qe^P}, which can be considered as a generalization
of Lemma 14 and Lemma 15 in \cite{LW}.
Since $\omega_l$ are arbitrary, this result is  more general than the corresponding
result in \cite{LW} even for the point case.

To describe operators appeared in the decomposition of $\exp(W_{\omega})$, we need introduce some notations first.
Let
\begin{equation}\label{Omega(z)}
B^{\omega}(z) := \sum\limits_{l=1}^{\infty} -\omega_l z^{-2l+1} .
\end{equation}
The minus sign before $\omega_l $ in the above equation is added in order to keep consistent with the point case in \cite{LW}.
  Define $R_i({\omega})$  by
$$e^{B^{\omega}(z)} = \sum\limits_{i=0}^{\infty} R_i({\omega}) z^{-i}.$$
Then $R_0 ({\omega})= 1,$ and for $i\geq 1$
\begin{equation}\label{Omega_i}
R_i ({\omega}) = \sum\limits_{n=1}^{i} \frac{1}{n!} \sum_{\substack{l_1,\dots l_n\geq 1 \\ \sum_{j=1}^n 2l_j=i+n} } (\prod\limits_{j=1}^{n}-\omega_{l_j}) .
\end{equation}
Define  $$Q_{\mathbf{\omega}}(x,y) := \frac{1-\exp(B^{\omega}(x^{-1}) + B^{\omega}(y^{-1}) )}{x+y}.$$
Using Equation~\eqref{Omega(z)}, we can expand $Q_{\mathbf{\omega}}(x,y)$ as
\begin{eqnarray}
Q_{\mathbf{\omega}}(x,y) = && \sum\limits_{n=1}^{\infty} \frac{1}{n!} \sum\limits_{l_1,\cdots l_n \geq 1} (\prod\limits_{m=1}^{n} -\omega_{l_m}) \sum\limits_{\substack{ i,j\geq 0 \\ i+j = 2l_1-2 }} (-1)^{i+1}  \nonumber \\
&& \hspace{20pt} \cdot \sum\limits_{k=0}^{n-1} \binom{n-1}{k} x^{i+\sum_{m=2}^{k+1}(2l_m-1)} y^{j+\sum_{m=k+2}^{n}(2l_m-1)}. \label{Q^Omega}
\end{eqnarray}
The derivation of this formula is the same as expanding the series $ Q^B(x,y) $ in \cite{LW} (cf. computation at the end of the proof of Lemma 15 in \cite{LW}). Observe that $ Q^B $ is related to $Q_{\mathbf{\omega}}$ in the following way
\begin{equation} \label{QB}
  Q^B(x,y) =  Q_{\mathbf{\omega}}(x,y)\mid_{ \omega_l = -\frac{B_{2l}}{2l(2l-1)}  }.
\end{equation}

Define a linear map $\Theta: {\rm Pow}(x,y) \longrightarrow {\rm Diff}(\mathbf{t})$
 by
\begin{equation} \label{Theta}
 \Theta(x^iy^j) =  \sum\limits_{\mu,\nu}\eta^{\mu\nu} \frac{\partial}{\partial t^{\mu}_{i}} \frac{\partial}{\partial t^{\nu}_{j}}
\end{equation}
where ${\rm Pow}(x,y)$ is the space of formal power series in $x$ and $y$, and
$$ {\rm Diff}(\mathbf{t}):= \left\{ \left. \sum\limits_{i,j\geq 0} c_{ij} \sum\limits_{\mu,\nu}\eta^{\mu\nu} \frac{\partial}{\partial t^{\mu}_{i}} \frac{\partial}{\partial t^{\nu}_{j}} \right| c_{ij} \ {\rm are} \ {\rm constans} \right\}$$ is a subspace
of the space of second order differential operators in variables $t_n^{\alpha}$.

When restricted to the point case, the only choice of $\mu$ and $\nu$ is $0$, and $\Theta$ is reduced to $\Theta^{pt}$ where
\begin{equation} \label{Thetapt}
\Theta^{pt}(x^iy^j) = \frac{\partial}{\partial t_{i}} \frac{\partial}{\partial t_{j}} .
\end{equation}
Note that $\Theta^{pt}$ is the map denoted by $\Theta_2$ in \cite{LW}.

Write
$$W_{\omega} = \mathfrak{B}_{t,\omega} + W' + \frac{\hbar}{2}W'',$$
where
\begin{eqnarray}
\mathfrak{B}_{t,\omega} &=& - \sum\limits_{l=1}^{\infty} \omega_l \sum\limits_{i=0}^{\infty}\sum\limits_{\alpha=0}^N t^{\alpha}_i\frac{\partial}{\partial t^{\alpha}_{i+2l-1}} \quad \in \ \mathfrak{g}_1,
    \label{Btomega} \\
 W' &=& \sum\limits_{l=1}^{\infty} \omega_l \frac{\partial}{\partial t^0_{2l}} \quad \in \ \mathfrak{g}_2, \nonumber \\
 W'' &=& \sum\limits_{l=1}^{\infty} \omega_l \sum\limits_{\substack{ i,j\geq 0 \\ i+j = 2l-2 }} (-1)^{i} \sum_{\mu, \nu} \eta^{\mu\nu} \frac{\partial}{\partial t^{\mu}_{i}} \frac{\partial}{\partial t^{\nu}_{j}} \quad \in \ \mathfrak{g}_2'. \nonumber
\end{eqnarray}

\begin{proposition} \label{e^W=e^Be^Qe^P}
$$\exp(W_{\omega}) = \exp(\mathfrak{B}_{t,\omega}) \exp(\frac{\hbar}{2}Q^{W}_{t,\omega}) \exp(P_{t,\omega}),$$
where $P_{t,\omega} = -\sum\limits_{i=1}^{\infty} R_i({\omega}) \frac{\partial}{\partial t^0_{1+i}}$
 and $Q^{W}_{t,\omega} = \Theta(Q_{\omega}(x,y))$.
\end{proposition}
\begin{proof}
Note that $W'$ commutes with $W''$. As in \cite{LW}, using Zassenhaus formula, i.e. Equation~\eqref{Zassenhaus}, to decompose
$\exp(W_{\omega}) = \exp((\mathfrak{B}_{t,\omega} + \frac{\hbar}{2} W'') + W')$ first
and then decompose $\exp(\mathfrak{B}_{t,\omega} + \frac{\hbar}{2} W'')$, we have
$$\exp(W_{\omega}) = \exp(\mathfrak{B}_{t,\omega}) \exp(\frac{\hbar}{2}Q^{W}_{t,\omega}) \exp(P_{t,\omega}),$$
where $$P_{t,\omega} = \sum\limits_{n=1}^{\infty} \frac{(-1)^{n-1}}{n!} ad_{\mathfrak{B}_{t,\omega}}^{n-1} W',$$
and
\begin{equation} \label{QWtbracket}
Q^{W}_{t,\omega} = \sum\limits_{n=1}^{\infty} \frac{(-1)^{n-1}}{n!} ad_{\mathfrak{B}_{t,\omega}}^{n-1} W''.
\end{equation}
Since $$ad_{\mathfrak{B}_{t,\omega}} W' = \sum\limits_{l_1,l_2 \geq 1} \omega_{l_1}\omega_{l_2} \frac{\partial}{\partial t^0_{2l_1+2l_2-1}},$$
comparing coefficients and indices of $W'$ and $ad_{\mathfrak{B}_{t,\omega}} W'$, we can deduce that
the iterated bracket $ad_{\mathfrak{B}_{t,\omega}}^{n-1} W'$ has the form
$$ ad_{\mathfrak{B}_{t,\omega}}^{n-1} W' =  \sum\limits_{l_1,\cdots,l_n \geq 1} (\prod\limits_{j=1}^{n}\omega_{l_j}) \frac{\partial}{\partial t^0_{1+\sum_{j=1}^{n} (2l_j-1)}} .$$
Hence
\begin{eqnarray*}
P_{t,\omega} &=& \sum\limits_{n=1}^{\infty} \frac{(-1)^{n-1}}{n!} \sum\limits_{i=n}^{\infty} \sum_{\substack{l_1,\dots l_n\geq 1 \\ \sum_{j=1}^n 2l_j=i+n} } (\prod\limits_{j=1}^{n}\omega_{l_j}) \frac{\partial}{\partial t^0_{1+i}} \\
&=& -\sum\limits_{i=1}^{\infty} \{ \sum\limits_{n=1}^{i} \frac{1}{n!} \sum_{\substack{l_1,\dots l_n\geq 1 \\ \sum_{j=1}^n 2l_j=i+n} } (\prod\limits_{j=1}^{n}-\omega_{l_j}) \} \frac{\partial}{\partial t^0_{1+i}} \\
&=& -\sum\limits_{i=1}^{\infty} R_i({\omega}) \frac{\partial}{\partial t^0_{1+i}}
\end{eqnarray*}
where the last equality is due to Equation~\eqref{Omega_i}.

Similarly, since
$$ ad_{\mathfrak{B}_{t,\omega}} W''  = \sum\limits_{l_1,l_2 \geq 1} \omega_{l_1}\omega_{l_2} \sum\limits_{\substack{ i,j\geq 0 \\ i+j = 2l_1-2 }} (-1)^{i} \sum_{\mu, \nu} \eta^{\mu\nu} ( \frac{\partial}{\partial t^{\mu}_{i+2l_2-1}} \frac{\partial}{\partial t^{\nu}_{j}} + \frac{\partial}{\partial t^{\mu}_{i}} \frac{\partial}{\partial t^{\nu}_{j+2l_2-1}} ),$$
we have
\begin{eqnarray*}
ad_{\mathfrak{B}_{t,\omega}}^{n-1} W'' = && \sum\limits_{l_1,\cdots,l_n \geq 1} (\prod\limits_{m=1}^{n} \omega_{l_m}) \sum\limits_{\substack{ i,j\geq 0 \\ i+j = 2l_1-2 }} (-1)^{i} \sum\limits_{k=0}^{n-1} \binom{n-1}{k}   \\
&& \hspace{60pt} \cdot  \sum_{\mu, \nu} \eta^{\mu\nu} \frac{\partial}{\partial t^{\mu}_{i+\sum_{m=2}^{k+1}(2l_m-1)}} \frac{\partial}{\partial t^{\nu}_{j+\sum_{m=k+2}^{n} (2l_m-1) }}.
\end{eqnarray*}
Combining with equations~\eqref{QWtbracket} and \eqref{Q^Omega}, we have
$$ Q^{W}_{t,\omega} = \Theta(Q_{\omega}(x,y)) .$$
\end{proof}

Since $P_{t,\omega}$ commutes with $Q^{W}_{t,\omega}$, by Proposition \ref{e^W=e^Be^Qe^P} and Equation~\eqref{Z=e^WZ}, we have
\begin{equation}\label{ZH=e^Be^Pe^QZ}
Z_{\mathbb{E}}^X(\mathbf{t},\mathbf{s}) = \exp(\mathfrak{B}_{t,\omega}) \exp(P_{t,\omega}) \exp(\frac{\hbar}{2}Q^{W}_{t,\omega})\cdot Z^X
\end{equation}
where $\omega_l = \frac{B_{2l}}{(2l)!} s_{2l-1}$ for all $l \geq 1$.
We observe that the operator $\exp(\frac{\hbar}{2}Q^{W}_{t,\omega})$
coincides with Givental's quantization of the scalar operator $\exp(B^{\omega}(z^{-1}))$ (cf. Proposition 7.3 in \cite{Gi2}).
It is interesting to see that Givental's quantization arises naturally in the process of decomposing differential operators
using Zassenhaus formula. To understand the action of the operator $\exp(\mathfrak{B}_{t,\omega}) \exp(P_{t,\omega})$, we notice that
$\mathfrak{B}_{t,\omega}$ and $P_{t,\omega}$ are both first order differential operators. Therefore the action of $\exp(\mathfrak{B}_{t,\omega}) \exp(P_{t,\omega})$ behaves as changing of variables and we only need to
understand its action on coordinates $t_n^{\alpha}$.
\begin{lemma} \label{lem:e^Be^Pt}
\begin{equation}\label{e^Be^Pt}
\exp(\mathfrak{B}_{t,\omega})\exp(P_{t,\omega}) \cdot t_n^{\alpha} =
\left \{ \begin{array}{ll} \sum\limits_{i=0}^{n} R_i({\omega}) t^{0}_{n-i}-R_{n-1}({\omega}), & \alpha = 0 \ and \ n\geq 2, \\  \sum\limits_{i=0}^{n}R_i({\omega}) t^{\alpha}_{n-i}, & otherwise. \end{array} \right.
\end{equation}
\end{lemma}

\begin{proof}
Noticed that $\exp(P_{t,\omega})$ is a shift operator. For any $t_n^{\alpha}$,
\begin{equation}\label{e^P.t}
 \exp(P_{t,\omega}) \cdot t_n^{\alpha} =
\left \{ \begin{array}{cl} t_n^{0}-R_{n-1}({\omega}), & \alpha = 0 \ and \ n\geq 2, \\  t_n^{\alpha}, & otherwise. \end{array} \right.
\end{equation}
We will set $t_{n}^{\alpha}=0$ if $n <0$. Then
 $ \mathfrak{B}_{t,\omega} \cdot t_n^{\alpha} = \sum\limits_{l\geq 1} (-\omega_l) t^{\alpha}_{n-(2l-1)} $.
Using the Taylor expansion for $\exp(\mathfrak{B}_{t,\omega})$, we have
\begin{eqnarray*}
\exp(\mathfrak{B}_{t,\omega}) \cdot t_n^{\alpha}  &=& t_n^{\alpha} + \sum\limits_{m\geq 1}\frac{1}{m!} \sum\limits_{l_1,\cdots l_m\geq 1} (\prod\limits_{j=1}^{m}-\omega_{l_j}) t^{\alpha}_{n-\sum_j(2l_j-1)} \\
&=& t_n^{\alpha} + \sum\limits_{i=1}^{n} \left\{ \sum\limits_{m=1}^{i} \frac{1}{m!} \sum_{\substack{l_1,\dots l_m\geq 1 \\ \sum_{j=1}^m 2l_j=i+m} } (\prod\limits_{j=1}^{m}-\omega_{l_j}) \right\} t^{\alpha}_{n-i}.
\end{eqnarray*}
Comparing with Equation~\eqref{Omega_i}, we have
\begin{equation}\label{e^B.t}
\exp(\mathfrak{B}_{t,\omega}) \cdot t_n^{\alpha} = \sum\limits_{i=0}^{n}R_i({\omega}) t^{\alpha}_{n-i},
\end{equation}
The lemma then follows from equations \eqref{e^P.t} and \eqref{e^B.t}.
\end{proof}

Set
\begin{equation} \label{hatt}
\hat{t}_n^\alpha := \exp(\mathfrak{B}_{t,\omega})\exp(P_{t,\omega}) \cdot t_n^{\alpha}
\end{equation}
for all $n$ and $\alpha$.
This defines $\hat{\mathbf{t}}=(\hat{t}_n^\alpha)$ as a function of $\mathbf{t}=(t_n^\alpha)$ and $\omega$.
Let $Q^W_{\hat{t},\omega}$ be the operator obtained from $Q_{t,\omega}^W$ after replacing $t_n^{\alpha}$ by
 $\hat{t}_n^{\alpha}$.
Then Equation~\eqref{ZH=e^Be^Pe^QZ} can be rewritten in the following form
\begin{corollary} \label{cor}
\[Z^{X}_{\mathbb{E}}(\mathbf{t},\mathbf{s}) = [e^{\frac{\hbar}{2} Q^W_{\hat{t},\omega}} \cdot Z^{X}(\hat{\mathbf{t}}) ]_{\hat{\mathbf{t}} = \hat{\mathbf{t}}(\mathbf{t},\omega)},\]
 where $\omega_l = \frac{B_{2l}}{(2l)!} s_{2l-1}$ for all $l \geq 1$.
\end{corollary}

This corollary can be regarded as a generalization of a lemma of Givental (cf. the lemma in Section~2.3 in \cite{Gi1}). To see this, we need the following
property:
\begin{lemma}
The two sets of variables $\mathbf{t}=(t_n^\alpha)$ and $\hat{\mathbf{t}}=(\hat{t}_n^\alpha)$ satisfy the following identity in  $H^{*}(X)[[z]]$:
\begin{equation}\label{hatT=T}
z \gamma_0 +\sum\limits_{n \geq 0} (-1)^n  z^n \sum\limits_{\alpha} \hat{t}_n^{\alpha}\gamma_{\alpha}
= \left(z \gamma_0 +\sum\limits_{n \geq 0} (-1)^n z^n \sum\limits_{\alpha} t_n^{\alpha}\gamma_{\alpha}\right)\exp(-B^{\omega}(z^{-1})) .
\end{equation}
\end{lemma}
\begin{proof}
Since
\[ \exp(-B^{\omega}(z^{-1})) = \exp(B^{\omega}((-z)^{-1})) = \sum_{n=0}^{\infty} R_n(\omega) (-z)^n,
\]
the right hand side of Equation~\eqref{hatT=T} has the following expansion
\[
 \gamma_0 \sum_{n=1}^{\infty} (-1)^{n-1} R_{n-1}(\omega) z^n +
  \sum\limits_{n\geq 0} \sum\limits_{i=0}^{n}(-1)^n R_{i} (\omega) \sum\limits_{\alpha}  t_{n-i}^{\alpha} \gamma_{\alpha} z^n .
\]
Note that $R_0(\omega)=1$. Since $\{ \gamma_{\alpha} \mid \alpha=0, 1, \ldots, N\}$ is a basis of $H^*(X)$,
comparing coefficients with the left hand side of Equation~\eqref{hatT=T}, we can see that Equation~\eqref{hatT=T}
is equivalent to
\begin{eqnarray*}
\widehat{t}^{\alpha}_0 &=& t^{\alpha}_0 , \\
\widehat{t}^{\alpha}_1 &=& R_1(\omega) t^{\alpha}_0+t^{\alpha}_1 , \\
\widehat{t}_n^{0} &=& \sum\limits_{i=0}^{n} R_i (\omega) t_{n-i}^0- R_{n-1}(\omega) \quad {\rm for \,\,\, } n\geq 2,\\
\widehat{t}_n^{\alpha} &=& \sum\limits_{i=0}^{n} R_{i}(\omega) t_{n-i}^{\alpha} \quad {\rm for \,\,\, } n\geq 2 \,\,\,
                   {\rm and \,\,\, }\alpha \neq 0.
\end{eqnarray*}
Therefore this lemma is equivalent to Lemma \ref{lem:e^Be^Pt}.
\end{proof}

{\bf Remark}:
In \cite{Gi1}, Givental considered the case when $X$ is a point and defined the change of
variables
$\mathbf{t}=(t_n) \longrightarrow \hat{\mathbf{t}} = (\hat{t}_{n})$
using Equation~\eqref{hatT=T}. Hence the lemma in Section 2.3 of \cite{Gi1} is the same
as the special case of Corollary~\ref{cor} for $X$ being a point.  Givental established this lemma by
arguing that the right hand side of the equation in Corollary~\ref{cor} also satisfies
the system of differential equations \eqref{DZ=0} in this special case. Our method in proving
Corollary~\ref{cor} is completely different from Givental's approach.

\vspace{10pt}

When dealing with linear Hodge integrals, we need perform the substitution
\begin{equation} \label{omega->u}
 \omega_l = -\frac{B_{2l}}{2l(2l-1)} u^{2(2l-1)}
 \end{equation}
for all $l \geq 1$. Under this substitution, the operator $W_{\omega}$ defined by equation \eqref{denW} becomes the operator $W_u$ defined by equation
\eqref{eqn:Wu}.
Let $\mathfrak{B}_{t,u}$, $Q^{W}_{t,u}$, $P_{t,u}$ be operators obtained from
$\mathfrak{B}_{t,\omega}, Q^{W}_{t,\omega}, P_{t,\omega}$   respectively after the substitution given in equation \eqref{omega->u} .
In this case, Proposition \ref{e^W=e^Be^Qe^P} becomes
\begin{corollary}
$$ \exp(W_u) = \exp(\mathfrak{B}_{t,u}) \exp(\frac{\hbar}{2}Q^{W}_{t,u}) \exp(P_{t,u}) $$
and
\begin{equation}\label{Q^W_u=Theta(Q_u)}
Q^{W}_{t,u} = \Theta(Q_{u}(x,y))
\end{equation}
where
\begin{equation} \label{Qu}
 Q_u(x,y) := Q_{\mathbf{\omega}}(x,y)\mid_{ \omega_l = -\frac{B_{2l}}{2l(2l-1)} u^{2(2l-1)} } .
 \end{equation}
\end{corollary}
Moreover, setting $s_{2l-1} = -(2l-2)!u^{2(2l-1)}$ in Equation~\eqref{ZH=e^Be^Pe^QZ},  we have
\begin{equation} \label{ZHBPQZt}
Z_{H}^X(\mathbf{t},u) = \exp(\mathfrak{B}_{t,u}) \exp(P_{t,u}) \exp(\frac{\hbar}{2}Q^{W}_{t,u}) \cdot Z^X(\mathbf{t}) .
\end{equation}

{\bf Remark}:
Since Proposition \ref{e^W=e^Be^Qe^P} holds for any value of $\omega_l$, we can use it to study other types of Hodge integrals too.
For example, the generating function of Hodge integrals with $k$ $\lambda$-classes is given by
$$F_{g,H,k}^{X}(\mathbf{t},u_1, \ldots, u_k) = \sum\limits_{A \in H_2(X,\mathbb{Z})} Q^{A}
    \langle \prod_{j=1}^k \lambda(-u_j^2) \exp (\sum\limits_{i=0}^{\infty}\sum\limits_{\alpha=0}^N t^{\alpha}_i \tau_i(\gamma_{\alpha})) \rangle_{g,A}^{X}.$$
By equation \eqref{lambdach}, $F_{g,H,k}^{X}$ can be obtained from $F_{g,\mathbb{E}}^{X}(\mathbf{t}, \mathbf{s})$ after the substitution
\[ s_{2l-1} = -(2l-2)! \left( u_1^{2(2l-1)} + \ldots + u_k^{2(2l-1)} \right) \]
for all $l \geq 1$.
Let $W_{u,k}$, $\mathfrak{B}_{t,u, k}$, $Q^{W}_{t,u,k}$, $P_{t,u,k}$ be operators obtained from
$W_{\omega}$, $\mathfrak{B}_{t,\omega}$, $Q^{W}_{t,\omega}$, $P_{t,\omega}$ respectively after the substitution
\[ \omega_l = -\frac{B_{2l}}{2l(2l-1)} \left( u_1^{2(2l-1)} + \ldots + u_k^{2(2l-1)} \right)  \]
for all $l \geq 1$. By equation \eqref{Z=e^WZ} and Proposition \ref{e^W=e^Be^Qe^P}, we have
\begin{eqnarray*}
 && \exp(\sum_{g=0}^{\infty} \hbar^{g-1} F_{g,H,k}^{X}(\mathbf{t},u_1,\ldots,u_k) ) = \exp(W_{u,k}) \cdot Z^X \\
&=& \exp(\mathfrak{B}_{t,u,k}) \exp(\frac{\hbar}{2}Q^{W}_{t,u,k}) \exp(P_{t,u,k}) \cdot Z^X .
\end{eqnarray*}
We will investigate applications of this formula in future publications.

\section{Substituition of Variables } \label{sec:change_variable}
Starting from this section, we will focus on the potential function $Z^{X}_{H}(\mathbf{t},u)$ for linear Hodge integrals.
We will change the variables $\mathbf{t}=(t_n^{\alpha})$  to variables $\mathbf{q}=(q_n^{\alpha})$.  Following \cite{LW}, we
introduce  a sequence of polynomials
\[ \phi_k(u,z) = \left( (u+z)^2z\frac{\partial}{\partial z} \right)^k z = \sum\limits_{j=1}^{2k+1}c_j^{(k)} u^{2k+1-j}z^j ,\]
where $c_j^{(k)}$ are some constants with the leading coefficient $c_{2k+1}^{(k)} = (2k-1)!!.$
In particular,  $\phi_0(u,z) = z.$ Replacing $z^m$ with $q_m$ in $\phi_k(u,z)$, we obtain the polynomial
$$ \tilde{\phi}_k(u,q) = \sum\limits_{j=1}^{2k+1}c_j^{(k)} u^{2k+1-j}q_j .$$
For any $0 \leq \alpha \leq N $, we define a map $\Lambda_{\alpha} $
from the space of functions of $(q_1, q_2, \ldots)$ to the space of functions of $\mathbf{q}=(q_n^{\alpha})$ which
 replaces $q_m$ by $q_m^{\alpha}$ for all $m \geq 1$.

 Moreover, the function $\mathfrak{B}(z)$ defined in \cite{LW} can be represented as
\begin{equation}
 \mathfrak{B}(z) = B^{\omega}(z)\mid_{ \omega_l = -\frac{B_{2l}}{2l(2l-1)} }.
 \end{equation}
Define constants $C_i$  by
\[ e^{\mathfrak{B}(z)} = \sum\limits_{i=0}^{\infty} C_i z^{-i}. \]
Then $C_0 = 1$ and
$$  C_i = R_i(\omega) \mid_{ \omega_l = -\frac{B_{2l}}{2l(2l-1)}  }$$
for all $i \geq 0$.
Since we need to perform substitution given by equation \eqref{omega->u} when studying linear Hodge integrals,
the following observation is
also useful:
\begin{equation} \label{BZ}
 \mathfrak{B}(z) = B^{\omega}(u^2 z)\mid_{ \omega_l = -\frac{B_{2l}}{2l(2l-1)} u^{2(2l-1)}}.
 \end{equation}
Using this equation, we can show that
\begin{equation} \label{CiRi}
u^{2i} C_i = R_i(\omega) \mid_{ \omega_l = -\frac{B_{2l}}{2l(2l-1)} u^{2(2l-1)} }
\end{equation}
and
\begin{equation} \label{QuQB}
Q_u(x, y) = u^2 Q^B(u^2 x, u^2 y),
\end{equation}
where $Q_u$ is defined by Equation \eqref{Qu} and $Q^B$ is defined by Equation \eqref{QB}.
Moreover the constants $C_i$ also satisfy the following identity (cf. Lemma 3 in \cite{LW}):
\begin{equation}\label{sumCC=0}
 \sum_{i=0}^n (-1)^{n-i}C_iC_{n-i} =0
\end{equation}
for $n\geq1$.

Define a  substitution of variables $ \mathbf{t}\rightarrow (u,\mathbf{q}) $ by
\begin{equation}\label{t->(u,q)}
t_n^{\alpha} =
\left \{ \begin{array}{cl} \Lambda_0(\tilde{\phi}_n(u,q)) + (-1)^{n}C_{n-1}u^{2(n-1)}, & \alpha = 0 \ and \ n\geq 2, \\
 \Lambda_{\alpha}(\tilde{\phi}_n(u,q)), & otherwise. \end{array} \right.
\end{equation}
In particular, if $u=0$, $ \mathbf{t}\rightarrow (0,\mathbf{q}) $ is given by a much simpler formula
 \[ t_k^{\alpha} = (2k-1)!!q_{2k+1}^{\alpha}.\]
We denote $$ Z_H^X (\mathbf{q},u) := Z_H^X (\mathbf{t},u)\mid_{\mathbf{t}\rightarrow (u,\mathbf{q})} ,$$
and $$ Z^X (\mathbf{q}) := Z_H^X (\mathbf{q},0) =  Z^X (\mathbf{t})\mid_{t_k^{\alpha} = (2k-1)!!q_{2k+1}^{\alpha} } .$$

\section{Decomposition of Virasoro operators}
\label{sec:decompVir}

Let $L_m$, $m \geq 1$, be the Virasoro operators defined by Equation~\eqref{Lm}. We can write $L_m$  as
\[ L_m = X_m + \frac{\hbar}{2} Y_m \]
 where
\[ X_m = \sum\limits_{k>0} \sum_{\alpha} (k+m)q_k^{\alpha}\frac{\partial}{\partial q_{k+m}^{\alpha}}, \quad
Y_m = \sum\limits_{a+b=m} \sum_{\mu, \nu} ab \eta^{\mu\nu} \frac{\partial^2}{\partial q_a^{\mu}\partial q_b^{\nu}}. \]
For convenience, we set $q_{n}^{\alpha}=0$ and $\frac{\partial}{\partial q_{n}^{\alpha}} =0$ if $n \leq 0$.
Let $$X_+ = \sum\limits_{m>0} a_mu^mX_m, \quad Y_+ = \sum\limits_{m>0} a_mu^mY_m,$$
where $a_m$ is defined by Equation~\eqref{a_m}.

When the projective variety $X$ is a point, $L_m $ is reduced to be
$$L_{m}^{pt} = \sum\limits_{k>0} (k+m)q_k\frac{\partial}{\partial q_{k+m}} + \frac{\hbar}{2}\sum\limits_{a+b=m} ab \frac{\partial^2}{\partial q_a\partial q_b}$$
for $m>0$,
and $X_m, Y_m, X_+, Y_+$ are reduced to be $X_{m}^{pt}, Y_{m}^{pt}, X^{pt}_+, Y^{pt}_+$ respectively.
It was proved in \cite{LW} that
\begin{equation} \label{e^Xq}
\exp(X^{pt}_+ )\cdot q_{2n+1} = \frac{1}{(2n-1)!!}\sum\limits_{i=0}^{n} C_iu^{2i} \tilde{\phi}_{n-i}(u,q).
 \end{equation}
 In fact, this formula follows from applying Proposition 5  in \cite{LW} to the case $G= \frac{t_n}{(2n-1)!!}$ and
 the computation before Equation (26) in \cite{LW}.

In \cite{LW}, the authors obtained the following decomposition using Zassenhaus formula:
$$\exp(\sum\limits_{m>0} a_mu^mL_{m}^{pt}) = \exp(X_{+}^{pt})\exp(\frac{\hbar}{2}Q^{pt}_+) ,$$
where $$Q^{pt}_+ = \sum\limits_{n=1}^{\infty} \frac{(-1)^{n-1}}{n!}ad_{X^{pt}_+}^{n-1}Y^{pt}_+$$
(cf. Proposition 7 in \cite{LW}).
The precise form of $Q^{pt}_+$ was also computed in \cite{LW}, which will not be needed in this paper. What we need is only Proposition 8 in \cite{LW} which can be stated in the following form:
\begin{equation} \label{Q+QW}
Q^{pt}_{+,odd} = Q_{q,u}^{W, pt},
\end{equation}
 where $Q^{pt}_{+,odd}$ is the sum of terms in $Q^{pt}_+$ involving only odd variables $q_{2k+1}$ and
 \begin{equation} \label{Q_q^W}
 Q_{q,u}^{W,pt} = \Theta^{pt}(Q_u(x,y))\mid_{t_k = (2k-1)!!q_{2k+1}}.
 \end{equation}
Note that the definition of $Q_{q,u}^{W,pt}$ given above coincide with the operator $Q_q^W$ in \cite{LW} due to Equation~\eqref{QuQB}.
It is also equal to the operator $Q_{t,u}^{W}$ defined by Equation~\eqref{Q^W_u=Theta(Q_u)}  after the substitution $t_k = (2k-1)!!q_{2k+1}$
in the case where the projective space $X$ is a point.

Similarly, since \[ \sum\limits_{m>0} a_m u^m L_m = X_+ + \frac{\hbar}{2} Y_+,\] we can use Zassenhaus formula to obtain a decomposition
\begin{equation}\label{e^L=e^Xe^Q}
\exp(\sum\limits_{m>0} a_m u^m L_m) = \exp(X_+)\exp(\frac{\hbar}{2}Q_+) ,
\end{equation}
where $$Q_+ = \sum\limits_{n=1}^{\infty} \frac{(-1)^{n-1}}{n!}ad_{X_+}^{n-1}Y_+. $$

Let ${\rm Diff}(\mathbf{q})$ be the space of second order differential operators in variables
$\mathbf{q}=(q_n^{\alpha})$ with constant coefficients.
Let ${\rm Diff}(\mathbf{q}^{pt})$ be the space of second order differential operators in variables
$\mathbf{q}^{pt}=(q_n)$ with constant coefficients.
Define a linear map \[ \Delta: {\rm Diff}(\mathbf{q}^{pt}) \longrightarrow {\rm Diff}(\mathbf{q}) \]  by
 \[ \Delta \left( \frac{\partial^2}{\partial q_m \partial q_n} \right) = \sum_{\mu, \nu} \eta^{\mu\nu} \frac{\partial^2}{\partial q_m^{\mu}\partial q_n^{\nu}}  \]
for all $m$ and $n$.
By induction on $k$, we have $$ad_{X_+}^{k}Y_+ = \Delta ( ad_{X^{pt}_+}^{k}Y^{pt}_+ ) $$ for all $k \geq 0$.
Consequently,
\[  Q_+ = \Delta (Q^{pt}_+) . \]
Let $ Q_{+,odd} $ be the sum of all terms in $Q_+$ which only involve odd variables $q_{2k+1}^{\alpha}$. We have
\begin{equation}\label{Q=Delta(Q)}
Q_{+,odd} = \Delta (Q^{pt}_{+,odd}) .
\end{equation}

\section{Proof of Theorem \ref{theorem}}
\label{sec:proofThm}

After performing the substitution of variables $\mathbf{t} \rightarrow (u, \mathbf{q})$ as given in Equation~\eqref{t->(u,q)},
Equation~\eqref{ZHBPQZt} becomes
\begin{equation} \label{ZHBPQZq}
Z_{H}^X(\mathbf{q},u) = \left\{ \left. \exp(\mathfrak{B}_{t,u}) \exp(P_{t,u}) \exp(\frac{\hbar}{2}Q^{W}_{t,u}) \cdot Z^X(\mathbf{t}) \right\} \right|_{\mathbf{t} \rightarrow (u, \mathbf{q})}.
\end{equation}
By Equation~\eqref{e^L=e^Xe^Q}, to prove Theorem \ref{theorem}, we need to compare the right hand side of this equation with
\begin{equation} \label{RHSThm}
 \exp(X_+)\exp(\frac{\hbar}{2}Q_+) \cdot \left\{ \left. Z^X (\mathbf{t}) \right|_{t_k^{\alpha} = (2k-1)!!q_{2k+1}^{\alpha}} \right\}.
\end{equation}
We need the following lemma:
\begin{lemma}\label{lemma}
For any formal power series $G$ in $\mathbf{t}$ and $u$,
\begin{equation}\label{e^Xe^PG=e^XG}
\left\{ \exp(\mathfrak{B}_{t,u}) \exp(P_{t,u}) \cdot G \right\}\mid_{\mathbf{t}\rightarrow (u,\mathbf{q})}
= \exp(X_+) \cdot \left\{ G\mid_{t_k^{\alpha} = (2k-1)!!q_{2k+1}^{\alpha}} \right\}.
\end{equation}
\end{lemma}
{\bf Remark}: This lemma is analogous to Proposition 5 in \cite{LW}. But they are different in the sense
 that the formula in \cite{LW} does not involve operator $\exp(P_{t,u})$.
For the point case, the action of $\exp(P_{t,u})$ on the Kontsevich-Witten tau-function
can be transformed into actions of the Virasoro operators due to the Virasoro constraints.
This is one ingredient in the proof of Alexandrov's conjecture in \cite{LW}.
However, this method fails in our case since we do not have the relevant form of Virasoro constraints for
 general projective varieties.  In this lemma, we include $\exp(P_{t,u})$ in the above formula in order to
 avoid using Virasoro constraints. Consequently we also need to choose a different substitution of variables
 to accomodate the appearance of this operator.

{\bf Proof of Lemma~\ref{lemma}}:
Notice that $ \mathfrak{B}_{t,u}, P_{t,u}, X_+ $ are all first order differential operators. So actions of
exponential of these operators on formal power series behave as change of variables.
Hence we only need to consider the case $G = t_n^{\alpha}$ for all $n$ and $\alpha$.
In this case, Equation~\eqref{e^Xe^PG=e^XG}
becomes
\begin{equation}\label{e^Xe^PG=e^Xtn}
\left\{ \exp(\mathfrak{B}_{t,u}) \exp(P_{t,u}) \cdot t_n^{\alpha} \right\}\mid_{\mathbf{t}\rightarrow (u,\mathbf{q})}
= (2n-1)!! \exp(X_+) \cdot q_{2n+1}^{\alpha} .
\end{equation}
We can compute the left hand side of this equation by setting $\omega_l = -\frac{B_{2l}}{2l(2l-1)} u^{2(2l-1)}$ for all $l$ in Equation~\eqref{e^Be^Pt}  and combining it with Equation~\eqref{CiRi}.

If $\alpha = 0 $ and $n\geq2$, the left hand side of Equation~\eqref{e^Xe^PG=e^Xtn} is
\begin{eqnarray*}
LHS &=& \left. \left\{ \sum\limits_{i=0}^{n}C_iu^{2i} t^{0}_{n-i}-C_{n-1}u^{2(n-1)} \right\}
    \right|_{\mathbf{t}\rightarrow (u,\mathbf{q})} \\
&=& \sum\limits_{i=0}^{n-2}C_iu^{2i} \left\{  \Lambda_0(\tilde{\phi}_{n-i}(u,q)) + (-1)^{n-i}C_{n-i-1}u^{2(n-i-1)} \right\} \\
&& + \sum\limits_{i=n-1}^{n}C_iu^{2i} \Lambda_0(\tilde{\phi}_{n-i}(u,q)) -C_{n-1}u^{2(n-1)} \\
&=& \sum\limits_{i=0}^{n}C_iu^{2i} \Lambda_0(\tilde{\phi}_{n-i}(u,q)) - \sum\limits_{i=0}^{n-2}(-1)^{n-1-i}C_iC_{n-1-i}u^{2(n-1)} -C_{n-1}u^{2(n-1)} \\
&=& \sum\limits_{i=0}^{n}C_iu^{2i} \Lambda_0(\tilde{\phi}_{n-i}(u,q)),
\end{eqnarray*}
where the last equality holds because of Equation~\eqref{sumCC=0}.

For all other $n$ and $\alpha$, we have
\[ LHS = \left. \left\{ \sum\limits_{i=0}^{n}C_iu^{2i} t^{\alpha}_{n-i} \right\}
    \right|_{\mathbf{t}\rightarrow (u,\mathbf{q})}
     = \sum\limits_{i=0}^{n}C_iu^{2i} \Lambda_{\alpha}(\tilde{\phi}_{n-i}(u,q)) . \]

To compute the right hand side of Equation~\eqref{e^Xe^PG=e^Xtn}, we first notice that
 \[ X_m \cdot q_{2n+1}^{\alpha} = (2n+1)q_{2n+1-m}^{\alpha} = \Lambda_{\alpha}((2n+1)q_{2n+1-m}) = \Lambda_{\alpha}(X^{pt}_{m} \cdot q_{2n+1}) \]
 for all $m$, $n$ and $\alpha$.
Therefore
\begin{eqnarray*}
&& ( X_{m_1} \cdots X_{m_k} ) \cdot q_{2n+1}^{\alpha} \\
&=& (2n+1)(2n+1-m_k)\cdots (2n+1-\sum\limits_{i=2}^{k}m_i)q^{\alpha}_{2n+1-m_1-\cdots-m_k}  \\
&=& \Lambda_{\alpha}( (X_{m_1}^{pt} \cdots X_{m_k}^{pt} ) \cdot q_{2n+1} ).
\end{eqnarray*}
Consequently, the right hand side of Equation~\eqref{e^Xe^PG=e^Xtn} is equal to
$$ RHS = (2n-1)!! \exp(X_+) \cdot q_{2n+1}^{\alpha} = (2n-1)!!\Lambda_{\alpha}(\exp(X^{pt}_+) \cdot q_{2n+1} ) .$$
Using $\Lambda_{\alpha}$ to act on  both sides of Equation~\eqref{e^Xq}, we can conclude
$$ LHS = RHS $$
for all $n$ and $\alpha$. The lemma is thus proved.
$\Box$

We are now ready to prove Theorem \ref{theorem}.

{\bf Proof of Theorem \ref{theorem}}:
By Equation~\eqref{ZHBPQZq} and Lemma \ref{lemma}, we have
\begin{eqnarray}
Z_H^X (\mathbf{q},u) &=& \exp(X_+) \left\{ \left. \left( \exp(\frac{\hbar}{2}Q^{W}_{t,u})\cdot Z^X(\mathbf{t}) \right) \right|_{t_k^{\alpha} = (2k-1)!!q_{2k+1}^{\alpha}} \right\}  \nonumber \\
&=& \exp(X_+) \exp(\frac{\hbar}{2}Q^{W}_{q,u})\cdot Z^X(\mathbf{q}) \label{ZHX+QZ}
\end{eqnarray}
where
\begin{eqnarray*}
Q^{W}_{q,u} &=& Q^{W}_{t,u} \mid_{t_k^{\alpha} = (2k-1)!!q_{2k+1}^{\alpha}} \\
&=& \Theta(Q_{u}(x,y))\mid_{t_k^{\alpha} = (2k-1)!!q_{2k+1}^{\alpha}} .
\end{eqnarray*}
The last equality comes from Equation~\eqref{Q^W_u=Theta(Q_u)}.

By Equation~\eqref{Q=Delta(Q)}, Equation~\eqref{Q+QW}, and Equation~\eqref{Q_q^W},  we have
\begin{eqnarray*}
Q_{+,odd} &=& \Delta(Q^{pt}_{+,odd}) = \Delta(Q^{W, pt}_{q,u}) \\
&=& \Delta\left(\Theta^{pt}(Q_{u}(x,y))\mid_{t_k = (2k-1)!!q_{2k+1}}\right) .
\end{eqnarray*}
For any monomial $x^iy^j$,
\begin{eqnarray*}
\Theta(x^iy^j)\mid_{t_k^{\alpha} = (2k-1)!!q_{2k+1}^{\alpha}} &=& \sum\limits_{\mu\nu}\frac{1}{(2i-1)!!(2j-1)!!} \eta^{\mu\nu}\frac{\partial}{\partial q^{\mu}_{2i+1}} \frac{\partial}{\partial q^{\nu}_{2j+1}} \\
&=& \Delta(\Theta^{pt}(x^iy^j)\mid_{t_k = (2k-1)!!q_{2k+1}}) .
\end{eqnarray*}
By linearity of $\Theta$, $\Theta^{pt}$, and $\Delta$,  we have
\[ \Theta(Q_{u}(x,y))\mid_{t_k^{\alpha} = (2k-1)!!q_{2k+1}^{\alpha}} =\Delta(\Theta^{pt}(Q_{u}(x,y))\mid_{t_k = (2k-1)!!q_{2k+1}}). \] Consequently,
$$Q^{W}_{q,u} = Q_{+,odd}.$$
Since $Z^X(\mathbf{q})$ only contains odd variables $q_{2k+1}^{\alpha}$,
$$ \exp(\frac{\hbar}{2}Q^{W}_{q,u}) \cdot Z^X(\mathbf{q}) = \exp(\frac{\hbar}{2}Q_+) \cdot Z^X(\mathbf{q}). $$
Therefore, by Equation~\eqref{ZHX+QZ} and Equation~\eqref{e^L=e^Xe^Q}, we have
\begin{eqnarray*}
Z_H^X (\mathbf{q},u) &=& \exp(X_+) \exp(\frac{\hbar}{2}Q_+) \cdot Z^X(\mathbf{q}) \\
&=& \exp(\sum\limits_{m>0} a_mu^mL_m) \cdot Z^X(\mathbf{q}).
\end{eqnarray*}
This completes the proof of Theorem~\ref{theorem}.
$\Box$

\appendix
\section{Virasoro relations}\label{app1}

In this appendix, we prove that operators $L_m = X_m + \frac{\hbar}{2}Y_m $ defined by Equation~\eqref{Lm} satisfy the Virasoro bracket
relation.

\begin{lemma}
  $$[ L_m, L_n ] = (m-n)L_{m+n} $$
  for all $m, n \geq 1$.
\end{lemma}
\begin{proof}
Since  $[Y_m,Y_n] = 0,$
$$ [ L_m, L_n ] = [X_m,X_n] + \frac{\hbar}{2}[X_m,Y_n] + \frac{\hbar}{2}[Y_m,X_n] .$$
It suffices if we can prove
\begin{equation} \label{bracketXX}
[X_m,X_n] = (m-n)X_{m+n}
\end{equation}
and
\begin{equation} \label{bracketXY}
[X_m,Y_n] + [Y_m,X_n] = (m-n)Y_{m+n}.
\end{equation}

Equation \eqref{bracketXX} follows from a straightforward computation. In fact
\begin{eqnarray*}
[X_m,X_n] &=& [ \sum\limits_{\alpha_1,k_1} (k_1+m) q_{k_1}^{\alpha_1}\frac{\partial}{\partial q_{k_1+m}^{\alpha_1}}, \sum\limits_{\alpha_2,k_2} (k_2+n) q_{k_2}^{\alpha_2}\frac{\partial}{\partial q_{k_2+n}^{\alpha_2}} ] \\
&=&  \sum\limits_{\alpha_1,k_1} (k_1+m)(k_1+m+n) q_{k_1}^{\alpha_1}\frac{\partial}{\partial q_{k_1+m+n}^{\alpha_1}}  \\
&& - \sum\limits_{\alpha_2,k_2} (k_2+n+m)(k_2+n) q_{k_2}^{\alpha_2}\frac{\partial}{\partial q_{k_2+n+m}^{\alpha_2}}  \\
&=& \sum\limits_{\alpha,k} (m-n)(k+m+n) q_{k}^{\alpha}\frac{\partial}{\partial q_{k+m+n}^{\alpha}}  \\
&=& (m-n)X_{m+n}.
\end{eqnarray*}

To prove Equation \eqref{bracketXY}, we first compute
\begin{eqnarray*}
[X_m,Y_n] &=&  [ \sum\limits_{\alpha,k} (k+m) q_{k}^{\alpha}\frac{\partial}{\partial q_{k+m}^{\alpha}}, \sum\limits_{a+b=n} ab \eta^{\mu\nu} \frac{\partial^2}{\partial q_a^{\mu}\partial q_b^{\nu}} ]\\
&=& -\sum_{a+b=n} ab(a+m) \eta^{\mu\nu}\frac{\partial}{\partial q_{a+m}^{\mu}} \frac{\partial}{\partial q_b^{\nu}}  - \sum_{a+b=n} ab(b+m) \eta^{\mu\nu}\frac{\partial}{\partial q_{b+m}^{\nu}} \frac{\partial}{\partial q_a^{\mu}}  \\
&=& \uppercase\expandafter{\romannumeral 1} + \uppercase\expandafter{\romannumeral 2} .
\end{eqnarray*}
where
$$ \uppercase\expandafter{\romannumeral 1}
= -\sum_{\substack{a+b=m+n \\ 1 \leq b \leq n-1 }} (a-m)ba \eta^{\mu\nu}\frac{\partial}{\partial q_{a}^{\mu}} \frac{\partial}{\partial q_b^{\nu}}
= \sum_{\substack{a+b=m+n \\ 1 \leq b \leq n-1 }} (b-n)ba \eta^{\mu\nu}\frac{\partial}{\partial q_{a}^{\mu}} \frac{\partial}{\partial q_b^{\nu}} ,$$
and
$$ \uppercase\expandafter{\romannumeral 2} = -\sum_{\substack{a+b=m+n \\ 1 \leq a \leq n-1 }} a(b-m)b \eta^{\mu\nu}\frac{\partial}{\partial q_{b}^{\nu}} \frac{\partial}{\partial q_a^{\mu}} = \sum_{\substack{a+b=m+n \\ m+1\leq b \leq m+n-1 }} a(m-b)b \eta^{\mu\nu}\frac{\partial}{\partial q_{b}^{\nu}} \frac{\partial}{\partial q_a^{\mu}} .$$
Similarly, $$[Y_m,X_n] = \uppercase\expandafter{\romannumeral 3} + \uppercase\expandafter{\romannumeral 4} $$
where
\begin{eqnarray*}
\uppercase\expandafter{\romannumeral 3} &=& \sum_{\substack{a+b=m+n \\ 1 \leq b \leq m-1 }} (a-n)ba \eta^{\mu\nu}\frac{\partial}{\partial q_{a}^{\mu}} \frac{\partial}{\partial q_b^{\nu}}  \\
&=& \sum_{\substack{a+b=m+n \\ 1 \leq b \leq m }} (m-b)ba \eta^{\mu\nu}\frac{\partial}{\partial q_{a}^{\mu}} \frac{\partial}{\partial q_b^{\nu}} ,
\end{eqnarray*}
and
\begin{eqnarray*}
\uppercase\expandafter{\romannumeral 4}
&=& \sum_{\substack{a+b=m+n \\ 1 \leq a \leq m-1 }} a(b-n)b \eta^{\mu\nu}\frac{\partial}{\partial q_{b}^{\nu}} \frac{\partial}{\partial q_a^{\mu}}  \\
&=& \sum_{\substack{a+b=m+n \\ n \leq b \leq  m+n-1 }} a(b-n)b \eta^{\mu\nu}\frac{\partial}{\partial q_{b}^{\nu}} \frac{\partial}{\partial q_a^{\mu}} .
\end{eqnarray*}
Hence
 $$ \uppercase\expandafter{\romannumeral 2} + \uppercase\expandafter{\romannumeral 3} = \sum_{a+b=m+n} (m-b)ab \eta^{\mu\nu}\frac{\partial}{\partial q_{a}^{\mu}} \frac{\partial}{\partial q_b^{\nu}} ,$$
$$ \uppercase\expandafter{\romannumeral 1} + \uppercase\expandafter{\romannumeral 4} = \sum_{a+b=m+n} (b-n)ab \eta^{\mu\nu}\frac{\partial}{\partial q_{a}^{\mu}} \frac{\partial}{\partial q_b^{\nu}} .$$
Adding these two formulas together, we get Equation \eqref{bracketXY}.
\end{proof}



\vspace{30pt} \noindent
Xiaobo Liu \\
School of Mathematical Sciences \& \\
Beijing International Center for Mathematical Research, \\
Peking University, Beijing, China. \\
Email: {\it xbliu@math.pku.edu.cn}

\vspace{30pt} \noindent
Haijiang Yu \\
Beijing International Center for Mathematical Research, \\
Peking University, Beijing, China. \\
Email: {\it yhjmath@pku.edu.cn}

\end{document}